\documentclass[a4paper,12pt]{amsart}

\usepackage{amsfonts}
\usepackage{amsmath}
\usepackage{amssymb}
\usepackage{graphicx}
\usepackage{mathrsfs}
\usepackage{hyperref}

\theoremstyle{plain}
 \newtheorem{theorem}{Theorem}[section]
 
 \newtheorem{lemma}{Lemma}[section]
 \newtheorem{property}{Property}[section]
 \newtheorem{corollary}{Corollary}[section]
\theoremstyle{definition}
 
 \newtheorem{definition}{Definition}[section]
\theoremstyle{remark}
 
 \numberwithin{equation}{section}

\title[Maximum principle and its application]{Maximum principle and its application for the nonlinear time-fractional diffusion equations with Cauchy-Dirichlet conditions}

\subjclass[2010]{Primary 26A33}

\keywords{sub-diffusion equation, maximum principle, Atangana-Baleanu derivative,
fractional differential equation, nonlinear problem, Riemann--Liouville derivative.}

\author[Borikhanov]{\bfseries Meiirkhan~Borikhanov}

\author[Kirane]{\bfseries Mokhtar~Kirane}

\author[Torebek]{\bfseries Berikbol T. Torebek$^{\text{*}}$}

\address{Meiirkhan~Borikhanov Al--Farabi Kazakh National University, Al--Farabi ave. 71, 050040, Almaty, Kazakhstan}
\email{meeir0808@gmail.com}

\address{Mokhtar~Kirane LaSIE, Facult\'{e} des Sciences, Pole Sciences et Technologies, Universit\'{e} de La Rochelle, Avenue M. Crepeau, 17042 La Rochelle Cedex, France \newline NAAM Research Group, Department of Mathematics, Faculty of Science, King Abdulaziz University, P.O. Box 80203, Jeddah 21589, Saudi Arabia \newline RUDN University, 6 Miklukho-Maklay St, Moscow 117198, Russia}
\email{mkirane@univ-lr.fr}

\address{Berikbol T. Torebek \newline
Institute of Mathematics and Mathematical Modeling. 125 Pushkin str., 050010 Almaty, Kazakhstan \newline Al--Farabi Kazakh National University, Al--Farabi ave. 71, 050040, Almaty, Kazakhstan}
\email{torebek@math.kz, btorebek@mail.kz}

\thanks{$\text{*}$ Corresponding author. E-mail: torebek@math.kz, btorebek@mail.kz}

\begin{document}

\vspace{18mm} \setcounter{page}{1} \thispagestyle{empty}

\begin{abstract}
In this paper, a maximum principle for the one-dimensional sub-diffusion equation with Atangana-Baleanu fractional derivative is formulated and proved. The proof of the maximum principle is based on an extremum principle for the Atangana-Baleanu fractional derivative that is given in the paper, too. The maximum principle is then applied to show that the initial-boundary-value problem for the linear and nonlinear time-fractional diffusion equations possesses at most one classical solution and this solution continuously depends on the initial and boundary conditions.
\end{abstract}

\maketitle
\section{Introduction and statement of problem}
In  this paper, we  consider the nonlinear time-fractional diffusion equation
\begin{equation}\label{1.1}\frac{\partial{u}}{\partial t}(x,t)=\frac{{{\partial }^{2}}}{\partial {{x}^{2}}}D_{*t}^{1-\alpha }u(x,t)+F\left( x,t, u\right) \text{ }in\text{ } \left( 0,a \right)\times \left( 0,T \right]=\Omega,\end{equation}
with the following nonhomogeneous Cauchy-Dirichlet conditions
\begin{equation}\label{1.2}\left\{\begin{array}{l}
u\left( x,0 \right)=\varphi ( x ),\,x \in\left[ 0,a \right],\\{}\\
u\left( 0,t \right)=\lambda(t),\,\, u\left( a,t \right)=\mu(t),\,  0\leq t\le T, \end{array}\right.
\end{equation}
where the functions $F\left( x,t, u\right),\varphi \left( x \right),\lambda \left( t \right),\mu \left( t \right)$ are continuous and ${\lambda }\left( t \right),{\mu }\left( t \right)$ are nondecreasing functions, $D_{*t}^{\alpha}$ is the Atangana-Baleanu fractional derivative (see Section \ref{S2}).

The aim of this paper is to research the maximum principle for the nonlinear fractional diffusion equation \eqref{1.1}. We first introduce some recent related works. Maximum principles were given in \cite{Y. Luchko1, Y. Luchko2, M. Al-Refai, M. Al-Refai1, Mohammed Al-Refai} for the types of fractional diffusion equations different from \eqref{1.1}. For the maximum principles given in \cite{M. Al-Refai} to hold, existence of a regular solution (with existence of a solutions ${{u}_{t}}$ on the closed time interval $\left[ 0,T \right]$) is assumed. In \cite{M. Al-Refai1}, the assumption of a solution with existence of a continuous ${{u}_{t}}$ in $\left( 0,T \right]$ such that ${{u}_{t}}\in {{L}^{1}}\left([ 0,T ]\right)$ is made. In \cite{Alsaedi} Ahmad, Alsaedi and Kirane studied three types of fractional diffusion equations. For each type, they obtained an upper bound of the Chebyshev norm in terms of the integral of the solution.

In \cite{Chan} Chan and Lui was considered a maximum principle for the equation \eqref{1.1} where  Riemann-Liouville derivative is considered rather that instead Atangana-Baleanu derivative \cite{Atangana, Atangana1, Baleanu}.

If $\alpha \to 0$ then equation \eqref{1.1} by Property \ref{p7} coincides with the classical heat equation. The equation of the form \eqref{1.1} with fractional derivatives with respect to the time variable is called the sub-diffusion equation \cite{Uchaikin}. This equation describes the slow diffusion.

The nonlinear problem \eqref{1.1} and \eqref{1.2} having a solution implies ${{u}_{t}}\left( x,t \right)$ exists. Thus for any $0<\alpha <1.$ $D_{t}^{1-\alpha }u\left( x,t \right)$ exists for $t>0.$ Hence, a solution $u\left( x,t \right)$ of the problem \eqref{1.1} and \eqref{1.2} in the region $\left[ 0,a \right]\times \left[ 0,T \right]$ is a (classical) solution in $C\left( \left[ 0,a \right]\times \left[ 0,T \right] \right)\cap {{C}^{2,1}}\left( \left( 0,a \right)\times \left( 0,T \right] \right).$

\section{Some definitions and properties of fractional operators}\label{S2}

In this section, we recall some basic definitions and properties of the fractional derivative operators.

\begin{definition}\label{d1} \cite{Kilbas} Let $f\in L_{loc}^1([a,b]),$ where $-\infty\leq a<t<b\leq+\infty$ be a locally integrable real-valued function. The Riemann--Liouville fractional integral $I^\alpha$ of order $\alpha\in\mathbb R$ ($\alpha>0$) is defined as
$$
I^\alpha  f\left( t \right) = \left(f*K_{\alpha}\right)(t) ={\rm{
}}\frac{1}{{\Gamma \left( \alpha \right)}}\int\limits_a^t {\left(
{t - s} \right)^{\alpha  - 1} f\left( s \right)} ds,$$
where $K_{\alpha}=\frac{t^{\alpha-1}}{\Gamma(\alpha)},$ $\Gamma$ denotes the Euler gamma function.
\end{definition}

\begin{definition}\cite{Kilbas} Let $f\in L^1([a,b]),$ $-\infty\leq a<t<b\leq+\infty$ and $f*K_{1-\alpha}\in W^1_{2}([a,b]), 0<\alpha<1$ where $W^{1}_2([a,b])$ is the Sobolev space. The Riemann--Liouville fractional derivative $D^\alpha$ of order $\alpha$ is defined as
$$D^\alpha f \left( t \right) = \frac{{d}}{{dt}}I^{1 - \alpha } f \left( t \right)={\rm{}}\frac{1}{{\Gamma \left( 1-\alpha \right)}}\frac{d}{dt}\int\limits_a^t {\left({t - s} \right)^{-\alpha} f\left( s \right)} ds.$$
\end{definition}

\begin{definition}\label{d2}\cite{Atangana} Let $0<\alpha <1$ and $f\in {{W}^{1}_2}\left([ a,b ] \right).$ The Atangana-Baleanu fractional derivative of order $\alpha $ is defined by

$${{D}_{*}}^{\alpha } f\left( t \right)=\frac{M(\alpha)}{1-\alpha }\int\limits_{a}^{t}{{f}'\left( s \right){{E}_{\alpha ,1}}\left[ -\alpha \frac{{{\left( t-s \right)}^{\alpha }}}{1-\alpha } \right]}ds,$$ where $M(\alpha)$ denotes a normalization function obeying $M(0) = M(1) = 1$ and $E_{\alpha,\beta} \left( z\right)$ is a Mittag-Leffler function $E_{\alpha,\beta} \left( z
\right)=\sum\limits_{k=0}^{\infty}\frac{z^k}{\Gamma(\alpha k+\beta)}.$
\end{definition}

\begin{definition}\label{d4}\cite{Atangana} Let $\alpha \ge 0$ and $f$ be the integrable function on $\left[ a,b \right].$ The Atangana-Baleanu fractional integral of order $\alpha $ is defined by
$${I}_{*}^{\alpha }\left[ f\left( t \right) \right]=\frac{1-\alpha}{M(\alpha)}f(t)+\frac{\alpha}{M(\alpha)}{I}^{\alpha }f(t).$$
\end{definition}
For convenience, in what follows we assume that $M(\alpha)=1.$
\begin{property}\label{p5} \cite{Atangana}Let $f\left( t \right)\in {{C}^{1}}\left([ a,b ]\right)$ and $\alpha \in \left( 0,1 \right),$ it holds
$${I}_{*}^{\alpha }{D}_{*}^{\alpha }\left[ f\left( t \right) \right]=f(t)-f(a).$$
\end{property}

\begin{property}\label{p6} If $f\left( t \right)\in {{C}^{1}}\left([ a,b ]\right)$, then it holds

\begin{align*}D_{*}^{\alpha }f\left( t \right)&=\frac{1}{1-\alpha }\left( f\left( t \right)-{{E}_{\alpha ,1}}\left[ -\alpha \frac{{{t}^{\alpha }}}{1-\alpha } \right]f\left( 0 \right) \right)\\&-\frac{\alpha }{{{\left( 1-\alpha  \right)}^{2}}}\int\limits_{0}^{t}{f\left( \tau  \right){{\left( t-\tau  \right)}^{\alpha -1}}{{E}_{\alpha ,\alpha }}\left[ -\alpha \frac{{{\left( t-\tau  \right)}^{\alpha }}}{1-\alpha } \right]}d\tau.\end{align*}
\end{property}
The property \ref{p6} is proved by applying integration by parts.

\begin{property}\label{p7} If $\alpha \to 0$, then  $$D_{*}^{\alpha }u\left( t \right)\to u\left( t \right)-u\left( 0 \right).$$
\end{property}

\begin{lemma}\label{l7.0}\cite{Mohammed Al-Refai}
Let $f\in {{C}^{1}}\left([ 0, T]\right)$ attain its maximum at ${{t}_{0}}\in \left( 0,\text{ }T \right)$, then $$D^{\alpha }f\left( {{t}_{0}} \right)\ge \frac{t_{0}^{-\alpha }}{\Gamma \left( 1-\alpha  \right)}f\left( {{t}_{0}} \right),\text{ }for\text{ }all\text{ }0<\alpha <1.$$
If, $f\left( {{t}_{0}} \right)\ge 0$, then $D^{\alpha }f\left( {{t}_{0}} \right)\ge 0$.
\end{lemma}
Analogous result for the fractional derivatives at absolute minimum points are obtained by applying the above result on $-f(t)$.

\begin{lemma}\label{l7}Let a function $f\left( t \right)\in {{C}^{1}}\left([0, T ]\right)$. Assume that ${f}'\left( t \right)$ exists and is continuous for $t\in \left[ 0,T \right].$\\
a)	If $f\left( t \right)$ attains its maximum value over $\left[ 0,T \right]$ at a point ${{t}_{0}}\in \left[ 0,T \right]$, then for $0<\alpha <1,$ we get
\begin{equation}\label{2.1}\left( D_{*}^{\alpha }f \right)\left[ {{t}_{0}} \right]\ge \frac{1}{1-\alpha }{{E}_{\alpha ,1}}\left[ -\alpha \frac{{{t}_{0}}^{\alpha }}{1-\alpha } \right]\left( f\left( {{t}_{0}} \right)-f\left( 0 \right) \right)\ge 0.
\end{equation}
b)	If $f\left( t \right)$ attains its minimum value over $\left[ 0,T \right]$ at a point ${{t}_{0}}\in \left[ 0,T \right]$, then for $0<\alpha <1,$ we have
\begin{equation}\label{2.1}\left( D_{*}^{\alpha }f \right)\left[ {{t}_{0}} \right]\le \frac{1}{1-\alpha }{{E}_{\alpha ,1}}\left[ -\alpha \frac{{{t}_{0}}^{\alpha }}{1-\alpha } \right]\left( f\left( {{t}_{0}} \right)-f\left( 0 \right) \right)\le 0.
\end{equation}
\end{lemma}

\begin{proof}
For the proof of part a) of Lemma \ref{l7} we define the auxiliary function $$g\left( t \right)=f\left( {{t}_{0}} \right)-f\left( t \right),\text{ }t\in \left[ 0,T \right].$$ Then it follows that $g\left( t \right)\ge 0,$ on $\left[ 0,T \right],$ $g\left( {{t}_{0}} \right)={g}'\left( {{t}_{0}} \right)=0$ and $$\left( D_{*}^{\alpha }g \right)\left[ t \right]=-\left( D_{*}^{\alpha }f \right)\left[ t \right].$$
Since $g\in {{C}^{1}}\left( 0,T \right),$ then ${g}'$ is integrable.  Property \ref{p6} and integrating by parts yields
\begin{align*}\left( D_{*}^{\alpha} g \right)\left( {{t}_{0}} \right)&
=\frac{1}{1-\alpha }\left( g\left( {{t}_{0}} \right)-{{E}_{\alpha ,1}}\left[ -\alpha \frac{{{t}_{0}}^{\alpha }}{1-\alpha } \right]g\left( 0 \right) \right)\\&
-\frac{\alpha }{{{\left( 1-\alpha  \right)}^{2}}}\int\limits_{0}^{{{t}_{0}}}{g\left( \tau  \right){{\left( t-\tau  \right)}^{\alpha -1}}{{E}_{\alpha ,\alpha }}\left[ -\alpha \frac{{{\left( {{t}_{0}}-\tau  \right)}^{\alpha }}}{1-\alpha } \right]}d\tau .\end{align*} Using the fact that ${{E}_{\alpha, 1}}(-\tau),\, \tau \in (0,\infty)$ is completely monotonic \cite{Miller} we have that \cite{Nieto} $${{E}_{\alpha ,\alpha }}(-\tau)>0,\, \tau \in (0,\infty).$$
Since $g\left( t \right)$ is nonnegative on $\left[ 0,T \right]$, the integral in the last equation is nonnegative, and thus
\begin{align*}\left( D_{*}^{\alpha} g \right)\left( {{t}_{0}} \right)&
\le \frac{1}{1-\alpha }\left( g\left( {{t}_{0}} \right)-{{E}_{\alpha ,1}}\left[ -\alpha \frac{{{t}_{0}}^{\alpha }}{1-\alpha } \right]g\left( 0 \right) \right)\\&
=-\frac{1}{1-\alpha }{{E}_{\alpha ,1}}\left[ -\alpha \frac{{{t}_{0}}^{\alpha }}{1-\alpha } \right]g\left( 0 \right)
=-\frac{1}{1-\alpha }{{E}_{\alpha ,1}}\left[ -\alpha \frac{{{t}_{0}}^{\alpha }}{1-\alpha } \right]\left( f\left( {{t}_{0}} \right)-f\left( 0 \right) \right).\end{align*}

The last inequality yields $$-\left( D_{*}^{\alpha }f \right)\left[ {{t}_{0}} \right]\le -\frac{1}{1-\alpha }{{E}_{\alpha ,1}}\left[ -\alpha \frac{{{t}_{0}}^{\alpha }}{1-\alpha } \right]\left( f\left( {{t}_{0}} \right)-f\left( 0 \right) \right),$$ which proves the result.
	By applying a similar argument to $-f\left( t \right)$, we obtain part b).
\end{proof}

\section{Linear time-fractional diffusion equation}
In this section we shall research the maximum principle for the linear case of equation \eqref{1.1}.

\begin{theorem}\label{t1} Let $u\left( x,t \right)$  satisfies the equation \begin{equation}\label{L_eq} \frac{\partial{u}}{\partial t}(x,t)=\frac{{{\partial }^{2}}}{\partial {{x}^{2}}}D_{*t}^{1-\alpha }u(x,t)+F\left(x,t\right) \text{ }in\text{ } \left( 0,a \right)\times \left( 0,T \right]=\Omega,\end{equation} with Cauchy-Dirichlet conditions \eqref{1.2} and functions $\lambda(t)$ and $\mu(t)$ are nondecreasing. If $F\left( x,t \right)\geq 0$ for $\left( x,t \right)\in \overline{\Omega },$ then $$u\left( x,t \right)\ge \underset{\left( x,t \right)\in \overline{\Omega }}{\mathop{\min }}\,\{\lambda \left( t \right),\mu \left( t \right),\varphi \left( x \right)\}\text{ for }\left( x,t \right)\in \overline{\Omega }.$$
\end{theorem}
\begin{proof}Let $m=\underset{\left( x,t \right)\in \overline{\Omega }}{\mathop{\min }}\,\{\lambda \left( t \right),\mu \left( t \right),\varphi \left( x \right)\}\text{ }$ and $\tilde{u}\left( x,t \right)=u\left( x,t \right)-m.$ Then, from \eqref{1.2} we obtain $\tilde{u}\left( 0,t \right)=\lambda \left( t \right)-m\ge 0,\,\tilde{u}\left( a,t \right)=\mu \left( t \right)-m\ge 0,\,\, t\in \left[ 0,T \right],$ and $\tilde{u}\left( x,0 \right)=\varphi \left( x \right)-m\ge 0,\,\,x\in \left[ 0,a \right].$

Since $\frac{\partial }{\partial t}\tilde{u}=\frac{\partial }{\partial t}u$ and $\frac{{{\partial }^{2}}}{\partial {{x}^{2}}}D_{*t}^{1-\alpha }\tilde{u}\left( x,t \right)=\frac{{{\partial }^{2}}}{\partial {{x}^{2}}}D_{*t}^{1-\alpha }u\left( x,t \right),$
it follows that $\tilde{u}\left( x,t \right)$ satisfies \eqref{L_eq}:
$${{\frac{\partial \tilde{u}}{\partial t}}}=\frac{{{\partial }^{2}}}{\partial {{x}^{2}}}D_{*t}^{1-\alpha }\tilde{u}\left( x,t \right)+F\left( x,t \right),$$
and initial-boundary conditions
$$\tilde{u}\left( x,0 \right)=\varphi \left( x \right)-m\geq 0,\,x\in \left[ 0,a \right],\,\,\tilde{u}\left( 0,t \right)=\lambda \left( t \right)-m\ge 0,\,\,\text{ }\tilde{u}\left( a,t \right)=\mu \left( t \right)-m \ge 0.$$

Suppose that there exits some $\left( x,t \right) \in \overline{\Omega }$  such that $\tilde{u}\left( x,t \right)$ is negative.
Since $$\tilde{u}\left( x,t \right)\ge 0,\,\,\left( x,t \right)\in \{0,a\}\times \left[ 0,T \right]\cup \left[ 0,a \right]\times \{0\},$$
there is $\left( {{x}_{0}},{{t}_{0}} \right) \in \Omega$  such that $\tilde{u}\left( {{x}_{0}},{{t}_{0}} \right)$ is the negative minimum of
$\tilde{u}$ over $\Omega.$ It follows from Lemma \ref{l7} that

\begin{multline}\label{3.1}
 D_{*t}^{\alpha }\tilde{u}\left( x,t \right)\le \frac{1}{1-\alpha }{{E}_{\alpha ,1}}\left[ -\alpha \frac{{{t}_{0}}^{\alpha }}{1-\alpha } \right]\left( \tilde{u}\left( {{x}_{0}},{{t}_{0}} \right)-\tilde{u}\left( {{x}_{0}},0 \right) \right)\\=
\frac{1}{1-\alpha }{{E}_{\alpha ,1}}\left[ -\alpha \frac{{{t}_{0}}^{\alpha }}{1-\alpha } \right]\left( \tilde{u}\left( {{x}_{0}},{{t}_{0}} \right)-\varphi \left( {{x}_{0}} \right) \right)\le \frac{1}{1-\alpha }{{E}_{\alpha ,1}}\left[ -\alpha \frac{{{t}_{0}}^{\alpha }}{1-\alpha } \right]\tilde{u}\left( {{x}_{0}},{{t}_{0}} \right)<0.
\end{multline}

Let $w\left( x,t \right)=D_{*t}^{1-\alpha }\tilde{u}\left( x,t \right)$. Since $\tilde{u}\left( x,t \right)$ is bounded in $\overline{\Omega },$ by Property \ref{p6} we have

\begin{multline}\label{3.2}
D_{*t}^{1-\alpha }\tilde{u}\left( x,t \right)=D_{*t}^{1-\alpha }u\left( x,t \right)
=\frac{1}{\alpha }\left( u\left( x,t \right)-{{E}_{1-\alpha ,1}}\left[ -\left( 1-\alpha  \right)\frac{{{t}^{1-\alpha }}}{\alpha } \right]u\left( x,0 \right) \right)\\-
\frac{1-\alpha }{{{\alpha }^{2}}}\int\limits_{0}^{t}{u\left(x, \tau  \right){{\left( t-\tau  \right)}^{-\alpha }}{{E}_{\alpha ,\alpha }}\left[ -\left( 1-\alpha  \right)\frac{{{\left( t-\tau  \right)}^{1-\alpha }}}{\alpha } \right]}d\tau \to 0\text{ as }t\to 0.
\end{multline}
From Property \ref{p5} we have that $\frac{\partial }{\partial t}I_{*t}^{1-\alpha }D_{*t}^{1-\alpha }\tilde{u}\left( x,t \right)=\frac{\partial }{\partial t}\tilde{u}\left( x,t \right).$
It follows from Definition \ref{d4} that
$${I}_{*}^{1-\alpha }w\left( x,t \right)=\alpha\cdot w\left( x,t \right)+(1-\alpha){I}_{*}^{1-\alpha }w\left( x,t \right).$$
Then, by using Definition \ref{d1} we get for any $t>0$, $$\frac{\partial }{\partial t}I_{*t}^{1-\alpha }w\left( x,t \right)=\alpha \frac{\partial{w}}{\partial t}\left( x,t \right)+\left( 1-\alpha  \right){{D}_{t}^{\alpha }}w\left( x,t \right).$$

It follows from a direct computation that by Property \ref{p6}
\begin{multline}\label{3.3}D_{*t}^{1-\alpha }u\left( x,t \right)=\frac{1}{\alpha }\left( u\left( x,t \right)-{{E}_{1-\alpha ,1}}\left[ -\left( 1-\alpha  \right)\frac{{{t}^{1-\alpha }}}{\alpha } \right]u\left( x,0 \right) \right)\\
-\frac{1-\alpha }{{{\alpha }^{2}}}\int\limits_{0}^{t}{u\left(x, \tau  \right){{\left( t-\tau  \right)}^{-\alpha }}{{E}_{\alpha ,\alpha }}\left[ -\left( 1-\alpha  \right)\frac{{{\left( t-\tau  \right)}^{1-\alpha }}}{\alpha } \right]}d\tau.
\end{multline}

Since the left-hand side of \eqref{3.3} and the first term of the right-hand side of \eqref{3.3} exist, it follows that the second term on the right-hand side exists and tends to 0 as $t\to {{0}^{+}}.$ Therefore, $D_{*t}^{1-\alpha }u\left( x, 0\right)=0$. Hence, we obtain
$$w\left( x,t \right)=D_{*t}^{1-\alpha }\tilde{u}\left( x,t \right)=D_{*t}^{1-\alpha }u\left( x,t \right)=0\text{ }as\text{ }t\to {{0}^{+}}.$$

Furthermore, it follows from the boundary condition of $\tilde{u}\left( x,t \right)$ that
$$D_{*t}^{1-\alpha }\tilde{u}\left( 0,t \right)=D_{*t}^{1-\alpha }\lambda\left(t \right),\,\,D_{*t}^{1-\alpha }\tilde{u}\left( a,t \right)=D_{*t}^{1-\alpha }\mu\left(t \right).$$ Since the functions $\lambda(t)$ and $\mu(t)$ are nondecreasing, then $\lambda'(t)\geq 0$ and $\mu'(t)\geq 0$ in $t\in(0,T).$ Consequently, $D_{*t}^{1-\alpha }\tilde{u}\left( 0,t \right)\geq 0$ and $D_{*t}^{1-\alpha }\tilde{u}\left( a,t \right)\geq 0.$

Therefore, $w\left( x,t \right)$ satisfies the problem
\begin{equation}\label{pr_w}\left\{\begin{array}{l}\alpha {\frac{\partial w}{\partial t}}\left( x,t \right)+\left( \alpha -1 \right){D_{t}^{\alpha }} w\left( x,t \right)={\frac{\partial^2 w}{\partial x^2}}\left( x,t \right)+F\left( x,t \right). \\{}\\
w\left( x,0 \right)=0, x\in\left[ 0,a \right], \\{}\\
w\left( 0,t \right)\ge 0,\text{ }w\left( a,t \right)\ge 0\ ,  0\le t\le T.\end{array}\right.\end{equation}

From \eqref{3.1} we have $w\left( {{x}_{0}},{{t}_{0}} \right)<0.$ Since $w\left( x,t \right)\geq 0$ on the boundary, there exists
$w\left( {{x}_{1}},{{t}_{1}} \right)\in \Omega$ such that $w\left( {{x}_{1}},{{t}_{1}} \right)$ is a negative minimum of $w\left( x,t \right)$ in $\overline{\Omega }$.

By Lemma \ref{l7.0}, $${D_{t}^{\alpha }} w\left( {{x}_{1}},{{t}_{1}} \right)\leq \frac{t_{1}^{-\alpha }}{\Gamma \left( 1-\alpha  \right)}w\left( {{x}_{1}},{{t}_{1}} \right)< 0.$$ Since $w\left( {{x}_{1}},{{t}_{1}} \right)$ is a local minimum, we obtain $\frac{\partial w}{\partial t} \left( {{x}_{1}},{{t}_{1}} \right)= 0$ and $\frac{\partial^2 w}{\partial x^2} \left( {{x}_{1}},{{t}_{1}} \right)\ge 0.$
Therefore at $\left( {{x}_{1}},{{t}_{1}} \right)$, we get ${{D}_{t}^{\alpha }}w\left(x_1, t_1 \right)< 0$ and ${{w}_{xx}}\left( {{x}_{1}},{{t}_{1}} \right)+F\left( {{x}_{1}},{{t}_{1}} \right)\ge 0.$ This contradiction shows that $\tilde{u}\left( x,t \right)\ge 0$ on $\overline{\Omega },$ and this implies that  $u\left( x,t \right)\ge m$ on $\overline{\Omega }$  for any $m$.
\end{proof}

\begin{theorem}\label{t2}
Suppose that $u\left( x,t \right)$  satisfies \eqref{1.1}, $u\left( x,0 \right)=\varphi \left( x \right)$ on $\left[ 0,a \right]$, $u\left( 0,t \right)=\lambda \left( t \right)$, $u\left( a,t \right)=\mu \left( t \right)$ and functions $\lambda(t)$ and $\mu(t)$ are nondecreasing. If $F\left( x,t \right)\le 0$ for $\left( x,t \right)\in \overline{\Omega },$ then $$u\left( x,t \right)\le \underset{\overline{\Omega }}{\mathop{\max }}\,\{\lambda \left( t \right),\mu \left( t \right),\varphi \left( x \right)\},\,\,\left( x,t \right)\in \overline{\Omega }.$$ \end{theorem}
Theorem \ref{3.1} and \ref{3.2} implies the following assertions
\begin{corollary} Suppose that $u\left( x,t \right)$  satisfies \eqref{1.1}, $u\left( x,0 \right)=0,\,x\in\left[ 0,a \right],$ $u\left( 0,t \right)=u\left( a,t \right)=0,\,t\in[0,T].$ If $F\left( x,t \right)\geq 0$ for $\left( x,t \right)\in \overline{\Omega },$ then $u\left( x,t \right)\geq 0,\,\,\left( x,t \right)\in \overline{\Omega }.$
\end{corollary}
\begin{corollary} Suppose that $u\left( x,t \right)$  satisfies \eqref{1.1}, $u\left( x,0 \right)=0,\,x\in\left[ 0,a \right],$ $u\left( 0,t \right)=u\left( a,t \right)=0,\,t\in[0,T].$ If $F\left( x,t \right)\leq 0$ for $\left( x,t \right)\in \overline{\Omega },$ then $u\left( x,t \right)\leq 0,\,\,\left( x,t \right)\in \overline{\Omega }.$
\end{corollary}
Theorems \ref{3.1} and \ref{3.2} are similar to the weak maximum principle for the heat equation. Similar to the classical case, the fractional version of the weak maximum principle can be used to prove the uniqueness of a solution.
\begin{theorem}\label{t3}
The problem \eqref{1.1}-\eqref{1.2} has at most one solution.
\end{theorem}
\begin{proof}
Let ${{u}_{1}}\left( x,t \right)$ and ${{u}_{2}}\left( x,t \right)$ be two solutions of the problem \eqref{1.1}-\eqref{1.2}. Then,
$$\frac{\partial }{\partial t}\left( {{u}_{1}}\left( x,t \right)-{{u}_{2}}\left( x,t \right) \right)=\frac{{{\partial }^{2}}}{\partial {{x}^{2}}}D_{*t}^{1-\alpha }\left( {{u}_{1}}\left( x,t \right)-{{u}_{2}}\left( x,t \right) \right),$$
with zero initial and boundary conditions for ${{u}_{1}}\left( x,t \right)-{{u}_{2}}\left( x,t \right)$. It follows from Theorems \ref{3.1} and \ref{3.2} that ${{u}_{1}}\left( x,t \right)-{{u}_{2}}\left( x,t \right)=0$ on $\overline{\Omega }$. We have a contradiction. The result then follows.\end{proof}
Theorems \ref{3.1} and \ref{3.2} can be used to show that a solution $u\left( x,t \right)$ of the problem \eqref{1.1}-\eqref{1.2} depends continuously on the initial data $\varphi \left( x \right).$

\begin{theorem}\label{t4}
Suppose $u\left( x,t \right)$ and $\overline{u}\left( x,t \right)$ are the solutions of the problem \eqref{1.1}-\eqref{1.2} with homogeneous boundary conditions
corresponding to the initial data $\varphi \left( x \right)$ and $\overline{\varphi }\left( x \right),$ respectively.

If $\underset{x\in \left[ 0,a \right]}{\mathop{\max }}\,
\{\left| \varphi \left( x \right)-\overline{\varphi }\left( x \right) \right|\}\le \delta ,$ then $\left| u\left( x,t \right)-\overline{u}\left( x,t \right) \right|\le \delta.$
\end{theorem}
\begin{proof}
The function $\widetilde{u}\left( x,t \right)=u\left( x,t \right)-\overline{u}\left( x,t \right)$ satisfies the problem $$\frac{\partial }{\partial t}\widetilde{u}\left( x,t \right)=\frac{{{\partial }^{2}}}{\partial {{x}^{2}}}D_{*t}^{1-\alpha }\widetilde{u}\left( x,t \right),$$ with initial condition $\widetilde{u}\left( x,0 \right)=\varphi \left( x \right)-\overline{\varphi }\left( x \right)$ and boundary conditions. It follows from Theorems \ref{3.1} and \ref{3.2} that
$$\left| \widetilde{u}\left( x,t \right) \right|\le \underset{\left[ 0,a \right]}{\mathop{\max }}\,\{\left| \varphi \left( x \right)-\overline{\varphi }\left( x \right) \right|\}.$$
The result then follows.
\end{proof}

\section{Nonlinear fractional diffusion equation}
We consider the nonlinear time-fractional diffusion equation of the form \eqref{1.1},
subject to the initial and boundary conditions
\begin{equation}\label{4.2}\left\{\begin{array}{l}
u\left( x,0 \right)=\varphi \left( x \right), x\in\left( 0,a \right), \\{}\\
u\left( 0,t \right)=u\left( a,t \right)=0,  0< t\le T, \end{array}\right.
\end{equation}
where $F(x,t,u)$ is a smooth function. We start with the following uniqueness result.

\begin{theorem}\label{4.1}
If $F(x,t,u)$ is nonincreasing with respect to $u$, then the nonlinear sub-diffusion equation \eqref{4.1} subject to the initial and boundary conditions \eqref{4.2} admits at most one solution $u\in {{C}^{2}}([0,a])\cap {{H}^{1}}((0,T])$
\end{theorem}
\begin{proof}
Assume that ${{u}_{1}}\left( x,t \right)$ and ${{u}_{2}}\left( x,t \right)$ are two solutions of \eqref{4.1} subject to initial and boundary conditions \eqref{4.2}, and let ${v}\left( x,t \right)={{u}_{1}}\left( x,t \right)-{{u}_{2}}\left( x,t \right).$ Then ${v}\left( x,t \right)$ satisfies
\begin{equation}\label{4.3}\left\{\begin{array}{l}
{{v}_{t}\left( x,t \right)}-\frac{{{\partial }^{2}}}{\partial {{x}^{2}}}D_{*t}^{1-\alpha }{{v}\left( x,t \right)}=F\left( x,t,{{u}_{2}} \right)-F\left( x,t,{{u}_{1}} \right), (x,t)\in\Omega,\\{}\\
{{v}\left( x,0 \right)}=0, 0< x < a,\\{}\\
{{v}\left( 0,t \right)}={{v}\left( a,t \right)}=0, t\in (0,T].\end{array}\right.
\end{equation}
Applying the mean value theorem to $F(x,t,u)$ yields
$$F\left( x,t,{{u}_{2}} \right)-F\left( x,t,{{u}_{1}} \right)=\frac{\partial F}{\partial u}\left( {{u}^{*}} \right)\left( {{u}_{2}}-{{u}_{1}} \right)=-\frac{\partial F}{\partial u}\left( {{u}^{*}} \right)v,$$
where $\left( {{u}^{*}} \right)=(1-\mu){{u}_{1}}+\mu{{u}_{2}}$ for some $0\leq \mu \leq1$. Thus,
$${{v}_{t}\left( x,t \right)}-\frac{{{\partial }^{2}}}{\partial {{x}^{2}}}D_{*t}^{1-\alpha }{{v}\left( x,t \right)}=-\frac{\partial F}{\partial u}\left( {{u}^{*}} \right)v.$$
Assume by contradiction that $v$ is not identically zero. Then $v$ has either a positive
 maximum or a negative minimum. At a positive maximum $\left( {{x}_{0}},{{t}_{0}} \right)\in {\Omega }$ and $F(x,t,u)$ is nonincreaing, we have $$\frac{\partial F}{\partial u}\left( {{u}^{*}} \right)\le 0\,\, \,\,\, and \,\,\,\,\,-\frac{\partial F}{\partial u}\left( {{u}^{*}} \right)v\left( {{x}_{0}},{{t}_{0}} \right)\ge 0,$$ then $${v}_{t}{\left( {{x}_{0}},{{t}_{0}} \right)}-\frac{\partial }{\partial {{x}^{2}}}D_{*t}^{1-\alpha }v\left( {{x}_{0}},{{t}_{0}} \right)\ge 0.$$

By using Theorem \ref{3.1} and Theorem \ref{3.2} for a positive maximum and a negative minimum respectively we get ${u}_{1}={u}_{2}.$
\end{proof}

\begin{theorem}\label{4.2}
If ${{u}_{1}}\left( x,t \right)$ and ${{u}_{2}}\left( x,t \right)$ are two solutions of the time-fractional diffusion equation \eqref{4.1} that satisfy the same boundary condition \eqref{4.2} and the initial conditions ${{u}_{1}}\left( x,0 \right)={g}_{1}(x)$ and ${{u}_{2}}\left( x,t \right)={g}_{2}(x)$, $0\leq x\leq a$. If $F(x,t,u)$ is nonincreasing with respect to $u$, then it holds that
$${{\left\| {{u}_{1}}\left( x,t \right)-{{u}_{2}}\left( x,t \right) \right\|}_{\overline{\Omega }}}\leq {{\left\| {{g}_{1}}\left( x \right)-{{g}_{2}}\left( x \right) \right\|}_{\left[ 0,a \right]}}.$$
\end{theorem}
\begin{proof}
Let ${v}\left( x,t \right)$=${u}_{1}-{u}_{2}$. Then ${v}\left( x,t \right)$ satisfies
\begin{equation}\label{4.6}\left\{\begin{array}{l}
{{v}_{t}\left( x,t \right)}-\frac{{{\partial }^{2}}}{\partial {{x}^{2}}}D_{*t}^{1-\alpha }{{v}\left( x,t \right)}=-\frac{\partial F}{\partial u}\left( {{u}^{*}} \right)v, (x,t)\in\Omega,\\{}\\
{{v}\left( x,0 \right)}={{g}_{1}}\left( x \right)-{{g}_{2}}\left( x \right), 0\leq x \leq a,\\{}\\
{{v}\left( 0,t \right)}={{v}\left( a,t \right)}=0, t\in (0,T].\end{array}\right.
\end{equation}
Let $\mathcal{M}={{\left\| {{g}_{1}}\left( x \right)-{{g}_{2}}\left( x \right) \right\|}_{\left[ 0,a \right]}},$ and assume by contradiction that the result of the Theorem \ref{4.2} is not true. That is, $$\|u_1-u_2\|_{\bar{\Omega}}\nleq \mathcal{M}.$$
Then $v$ either has a positive maximum at a point $\left( {{x}_{0}},{{t}_{0}} \right)\in {\Omega }$ with $$v\left( {{x}_{0}},{{t}_{0}} \right)=\mathcal{M}_1>\mathcal{M},$$ or it has a negative minimum at a point $\left( {{x}_{0}},{{t}_{0}} \right)\in {\Omega }$ with $$v\left( {{x}_{0}},{{t}_{0}} \right)=\mathcal{M}_2<-\mathcal{M}.$$ If $$v\left( {{x}_{0}},{{t}_{0}} \right)=\mathcal{M}_1>\mathcal{M},$$ using the initial and boundary conditions of $v$, we have $\left( {{x}_{0}},{{t}_{0}} \right)={{\Omega }_{T}}.$

Applying analogous Theorem \ref{3.1} and Theorem \ref{3.2} in the proof of the previous theorem, we have $\left\| v\left( x,t \right)\right \|\leq{\mathcal{M}},$ which proves the result.
\end{proof}

\section*{Acknowledgements} M. Kirane was supported by the Ministry of Education and
Science of the Russian Federation (Agreement number No 02.a03.21.0008). M. Borikhanov and B. T. Torebek was financially supported by a grant 2018-2020 from the Ministry of Science and Education of the Republic of Kazakhstan.

\end{document}